\documentclass{amsart}
\usepackage{amssymb,mathtools,dsfont,nicefrac}
\usepackage[all]{xy}
\usepackage{hyperref}
\usepackage[lite]{amsrefs}
\usepackage{enumitem}
\usepackage{microtype}

\newcommand*{\MRref}[2]{ \href{http://www.ams.org/mathscinet-getitem?mr=#1}{MR \textbf{#1}}}
\newcommand*{\arxiv}[1]{\href{http://www.arxiv.org/abs/#1}{arXiv: #1}}

\newtheorem{theorem}{Theorem}[section]

\newtheorem{lemma}[theorem]{Lemma}
\newtheorem{proposition}[theorem]{Proposition}

\theoremstyle{remark}
\newtheorem{example}[theorem]{Example}

\theoremstyle{definition}
\newtheorem{definition}[theorem]{Definition}

\DeclareMathOperator{\Smooth}{S}
\DeclareMathOperator{\Rough}{R}

\DeclareMathOperator{\Hom}{Hom}

\newcommand*{\nb}{\nobreakdash}
\newcommand*{\blank}{\textup{\textvisiblespace}}
\newcommand*{\defeq}{\mathrel{\vcentcolon=}}
\newcommand*{\eqdef}{\mathrel{=\vcentcolon}}
\newcommand*{\congto}{\xrightarrow\cong}
\newcommand*{\prot}{\mathbin{\hat\otimes_\pi}}
\newcommand*{\hot}{\mathbin{\hat\otimes}}

\newcommand*{\ihom}[3][]{#2\mathbin{\Rightarrow}_{#1}#3}
\newcommand*{\ihomr}[3][]{#2\mathbin{\Leftarrow}_{#1}#3}

\newcommand*{\N}{\mathbb N}
\newcommand*{\Z}{\mathbb Z}
\newcommand*{\C}{\mathbb C}
\newcommand*{\Mat}{\mathbb M}
\newcommand*{\Unit}{\mathds 1}

\newcommand*{\g}{\mathfrak g}
\newcommand*{\Leftm}{\mathcal M_\textup l}
\newcommand*{\Rightm}{\mathcal M_\textup r}

\newcommand*{\Cinf}[1][\infty]{\textup C^{#1}}
\newcommand*{\Ccinf}[1][\infty]{\textup C_\textup c^{#1}}
\newcommand*{\Rdk}{\mathcal K}
\newcommand*{\Sch}{\mathcal S}
\newcommand*{\Univ}{\textup U}

\newcommand*{\Cat}{\mathcal C}
\newcommand*{\Born}{\mathfrak{Bor}}
\newcommand*{\Top}{\mathfrak{Tvs}}
\newcommand*{\Indban}{\overrightarrow{\mathfrak{Ban}}}
\newcommand*{\Mod}[1]{\mathfrak{Mod}_{#1}}

\newcommand*{\Id}{\textup{Id}}
\newcommand*{\ev}{\textup{ev}}
\newcommand*{\diff}{\textup d}
\newcommand*{\op}{\textup{op}}

\begin{document}

\title[Smooth and rough modules over self-induced algebras]{Smooth and rough modules\\over self-induced algebras}
\author{Ralf Meyer}
\email{rameyer@uni-math.gwdg.de}
\address{Mathematisches Institut and Courant Centre ``Higher order structures''\\
  Georg-August Universit\"at G\"ottingen\\
  Bunsenstra{\ss}e 3--5\\
  37073 G\"ottingen\\
  Germany}

\begin{abstract}
  A non-unital algebra in a closed monoidal category is
  called self-induced if the multiplication induces an
  isomorphism \(A\otimes_A A \cong A\).  For such an algebra,
  we define smoothening and roughening functors that retract
  the category of modules onto two equivalent subcategories of
  smooth and rough modules, respectively.  These functors
  generalise previous constructions for group representations
  on bornological vector spaces.  We also study the pairs of
  adjoint functors between categories of smooth and rough
  modules that are induced by bimodules and Morita equivalence.
\end{abstract}
\thanks{Supported by the German Research Foundation (Deutsche Forschungsgemeinschaft (DFG)) through the Institutional Strategy of the University of G\"ottingen}
\maketitle

\section{Introduction}
\label{sec:intro}

Many algebras that are considered in non-commutative geometry
are non-unital.  Typical examples are the convolution algebra
\(\Ccinf(G)\) of smooth compactly supported functions on a
locally compact group~\(G\) (see~\cite{Meyer:Smooth}) or the
algebras \(\Mat_\infty\) and~\(\Rdk\) of finite matrices and of
infinite matrices with rapidly decreasing entries
(see~\cite{Cuntz-Meyer-Rosenberg}).  Both algebras carry
additional structure: both \(\Ccinf(G)\) and~\(\Rdk\) are
complete convex bornological algebras.  We may also
view~\(\Rdk\) as a Fr\'echet algebra, but this structure is
less relevant here.

When dealing with non-unital algebras, the usual unitality
condition for modules makes no sense.  But simply dropping this
condition would give too many modules.  On the one hand, the
bornological algebras \(\Mat_\infty\) and~\(\Rdk\) are Morita
equivalent to~\(\C\), so that we expect an equivalence of
module categories.  On the other hand, the categories of
non-unital (bornological) modules over \(\Mat_\infty\)
and~\(\Rdk\) are not equivalent to the category of
\(\C\)\nb-modules.

This article grew out of the
manuscript~\cite{Meyer:Embed_derived}, which will not be
published any more because there are too many small things that
I want changed.  One of them is that,
while~\cite{Meyer:Embed_derived} only considers bornological
algebras, it is sometimes necessary to consider other
categories instead of bornological vector spaces, such as the
category of inductive systems of Banach spaces
(see~\cite{Meyer:HLHA}).  Therefore, we discuss smoothening and
roughening functors and the functors induced by bimodules in
much greater generality here.  We work with algebras in an
arbitrary monoidal category, and we replace the quasi-unitality
assumption in~\cite{Meyer:Embed_derived} by the much weaker
assumption of being self-induced.  A monoidal category is an
additive category~\(\Cat\) with an associative tensor product
functor~\(\otimes\) and a tensor unit~\(\Unit\) that satisfies
suitable coherence laws (see \cites{MacLane:Associativity,
  Saavedra:Tannakiennes}).

Following Niels
Gr{\o}nb{\ae}k~\cite{Gronbaek:Morita_self-induced}, we call an
algebra~\(A\) in such a tensor category \emph{self-induced} if
the multiplication map \(A\otimes A\to A\) induces an
isomorphism \(A\otimes_A A \cong A\).  If~\(A\) is
self-induced, Gr{\o}nb{\ae}k calls a left \(A\)\nb-module~\(X\)
\emph{\(A\)\nb-induced} if the multiplication map \(A\otimes
X\to X\) induces an isomorphism \(A\otimes_A X \cong X\).

For instance, let~\(\Cat\) be the symmetric monoidal category
of complete convex bornological vector spaces with the complete
projective bornological tensor product and let \(A= \Ccinf(G)\)
be the convolution algebra of smooth functions with compact
support on a locally compact group~\(G\) (in the sense of
Fran\c{c}ois Bruhat~\cite{Bruhat:Distributions}), viewed as an
algebra in~\(\Cat\).  Then~\(A\) is self-induced, and the
category of \(A\)\nb-induced modules is isomorphic to the
category of smooth representations of~\(G\) on complete convex
bornological vector spaces (see~\cite{Meyer:Smooth}).  Therefore,
we call \(A\)\nb-induced modules \emph{smooth}.

An \(A\)\nb-module~\(X\) over a self-induced algebra~\(A\) is
called \emph{rough} if the adjoint \(X\to (\ihom AX)\) of the
module multiplication map \(A\otimes X\to X\) induces an
isomorphism \(X\cong \ihom[A]AX\).  Here \((\ihom AX) =
\Hom(A,X)\) denotes the internal Hom functor and \(\ihom[A]AX =
\Hom_A(A,X)\) denotes the subfunctor of ``\(A\)\nb-linear
maps.''  The existence of such internal Hom functors is the
defining property of a \emph{closed} monoidal category.  In the
category of complete convex bornological vector spaces, \(\ihom[A]AX\) is the
space of bounded \(A\)\nb-module homomorphisms \(A\to X\).

Rough and smooth modules and smoothening and roughening functors for
group convolution algebras are already studied in~\cite{Meyer:Smooth}.
Here we extend some of the properties observed in~\cite{Meyer:Smooth}
to the general setting explained above.  The smoothening and
the roughening of a module~\(X\) are defined by
\[
\Smooth(A) \defeq A\otimes_A X
\qquad\text{and}\qquad
\Rough(A) \defeq \ihom[A]AX,
\]
respectively.  As the name suggests, these \(A\)\nb-modules are
smooth and rough, respectively.  There are natural maps
\(\Smooth(X)\to X\to\Rough(X)\), the first is an isomorphism if
and only if~\(X\) is smooth, the second if and only if~\(X\) is
rough.  Thus \(\Smooth\) and~\(\Rough\) are retractions from
the category of all modules onto the subcategories of smooth
and rough modules, respectively.  We show also that~\(\Smooth\)
is the right adjoint of the embedding of the subcategory of
smooth modules, while~\(\Rough\) is left adjoint to the
embedding of the subcategory of rough modules.  And~\(\Rough\)
is right adjoint to~\(\Smooth\).  Finally,
\(\Smooth\circ\Rough=\Smooth\) and
\(\Rough\circ\Smooth=\Rough\), so that the functors \(\Smooth\)
and~\(\Rough\) provide an equivalence of categories between the
categories of smooth and rough \(A\)\nb-modules.

This is useful when we want to turn bimodules into functors
between categories of smooth or rough modules.  Of course, an
\(A,B\)-bimodule~\(M\) induces a functor \(X\mapsto M\otimes_B
X\) from left \(B\)- to left \(A\)\nb-modules.  If~\(M\) is
smooth as a left \(A\)\nb-module, this maps smooth modules
again to smooth modules.  The functor \(Y\mapsto \ihom[B]MY\)
in the opposite direction is defined between the categories of
rough modules, that is, \(\ihom[A]MY\) is a rough
\(B\)\nb-module if~\(M\) is a smooth \(B\)\nb-module.  Using
the smoothening functor, we may turn this into a functor
between categories of smooth modules as well.  The resulting
functor \(Y\mapsto \Smooth(\ihom[A]MY)\) is right adjoint to
the functor \(X\mapsto M\otimes_B X\) in the opposite
direction.  In particular, the functor \(X\mapsto M\otimes_B
X\) between smooth module categories always has a right adjoint
functor.

An algebra homomorphism \(f\colon A\to B\) allows us to
view~\(B\) as an \(A,B\)-bimodule or as a \(B,A\)-bimodule.
These two bimodules provide two pairs of adjoint functors
between the categories of smooth modules over \(A\) and~\(B\).

As an example of our general theory, we consider the biprojective
algebras of the form \(W\otimes V\) associated to a sufficiently
non-degenerate map \(b\colon V\otimes W\to\Unit\).  This includes the
bornological algebras \(\Mat_\infty\) and~\(\Rdk\) of finite and
rapidly decreasing matrices.  This construction also provides examples
of self-induced bornological algebras where the canonical map
\(\Smooth(X) \to X\) is not always a monomorphism.  This should be
contrasted with \cite{Meyer:Smooth}*{Lemma 4.4}, which asserts that
this map is always injective provided~\(A\) is a bornological algebra
with an approximate identity in a suitable sense.

We also consider the functors that relate Lie group and Lie algebra
representations for a Lie group~\(G\).  Let \(\Univ(\g)\) be the universal
enveloping algebra of the Lie algebra~\(\g\) of~\(G\).  Thus the
category of unital \(\Univ(\g)\)-modules is equivalent to the category of
Lie algebra representations of~\(\g\).  We may view \(\Ccinf(G)\) as a
\(\Ccinf(G),\Univ(\g)\)- or \(\Univ(\g),\Ccinf(G)\)-bimodule.  This provides
two functors from smooth representations of~\(G\) to Lie algebra
representations of~\(\g\).  The first equips a smooth representation
with the induced representation of~\(\g\), the second takes the
induced representation of~\(\g\) on the roughening.  In the opposite
direction, we get two functors that integrate representations of~\(\g\) to
smooth representations of~\(G\).

\section{Preliminaries}
\label{sec:preliminaries}

Additive monoidal categories provide the categorical framework
to define algebras and modules.  In the same generality, we may
define self-induced algebras and smooth modules.  We need a
\emph{closed} monoidal category, that is, an internal Hom
functor, to define rough modules as well.  Here we briefly
recall these basic category theoretic definitions.  Then we
turn the categories of Banach spaces, of complete convex
bornological vector spaces, and of inductive systems of Banach
spaces into closed monoidal categories.  We also discuss the
monoidal category of complete locally convex topological vector
spaces and why it is not closed.

Readers who are only interested in bornological and topological
algebras need not read this section in detail because
everything we explain here is fairly obvious in those cases.
They mainly have to remember the functors \(X\otimes_A Y\) and
\(\ihom[A]XY\) described concretely in
Example~\ref{exa:balanced_tensor_Born} and the basic
adjointness relation~\eqref{eq:adjoint_tensor_hom}.

A monoidal category is a category~\(\Cat\) with a bifunctor
\(\otimes\colon \Cat\times\Cat\to\Cat\) called \emph{tensor
  product} and an object~\(\Unit\) called (tensor) \emph{unit},
and natural isomorphisms
\[
\alpha_{A,B,C}\colon (A\otimes B)\otimes C\congto A\otimes
(B\otimes C),\qquad
\lambda_A\colon \Unit\otimes A\congto A,\quad
\rho_A\colon A\otimes\Unit\congto A,
\]
called \emph{associator}, \emph{left unitor} and \emph{right
  unitor}, subject to two coherence conditions: for all objects
\(A\), \(B\), \(C\) and~\(D\) in~\(\Cat\), the pentagon diagram
\[
\xymatrix@C+1.9em{
  \bigl((A\otimes B)\otimes C\bigr) \otimes D
  \ar[d]_{\alpha_{A\otimes B,C, D}}
  \ar[r]^{\alpha_{A,B,C}\otimes D}&
  \bigl(A\otimes (B\otimes C)\bigr) \otimes D
  \ar[r]^{\alpha_{A,B\otimes C, D}}&
  A\otimes \bigl((B\otimes C) \otimes D\bigr)
  \ar[d]^{A\otimes \alpha_{B, C, D}}\\
  (A\otimes B)\otimes (C\otimes D)
  \ar[rr]_{\alpha_{A,B,C\otimes D}}&&
  A\otimes \bigl(B\otimes (C \otimes D)\bigr)
}
\]
commutes, and for all objects \(A\), \(B\) and~\(C\)
of~\(\Cat\), the diagram
\[
\xymatrix@C-2em{
  (A\otimes\Unit)\otimes B \ar[rr]^{\alpha_{A,\Unit,B}}
  \ar[dr]_{\rho_A\otimes B}&&
  A\otimes(\Unit\otimes B) \ar[dl]^{A\otimes\lambda_B}\\
  &A\otimes B
}
\]
commutes.  By Mac Lane's Coherence
Theorem~\cite{MacLane:Associativity}, these two coherence
conditions imply that any diagram constructed using only
associators and unitors commutes.

A \emph{braided monoidal category}~\cite{Joyal-Street:Braided}
is a monoidal category together with \emph{braiding automorphisms}
\(\gamma_{A,B}\colon A\otimes B\to B\otimes A\) that are
compatible with the associators in the sense that the following
hexagons commute:
\[
\xymatrix@C+1em{
  &A\otimes(B\otimes C) \ar[r]^{\gamma_{A,B\otimes C}}&
  (B\otimes C)\otimes A \ar[dr]^{\alpha_{B,C,A}}&\\
  (A\otimes B)\otimes C\ar[ur]^{\alpha_{A,B,C}}
  \ar[dr]_{\gamma_{A,B}\otimes C}&&&B\otimes(C\otimes A)\\
  &(B\otimes A)\otimes C\ar[r]_{\alpha_{B,A,C}}&
  B\otimes (A\otimes C)\ar[ur]_{B\otimes\gamma_{A,C}}
}
\]
\[
\xymatrix@C+1em{
  &(A\otimes B)\otimes C \ar[r]^{\gamma_{A\otimes B,C}}&
  C\otimes(A\otimes B) \ar[dr]^{\alpha^{-1}_{C,A,B}}&\\
  A\otimes (B\otimes C)\ar[ur]^{\alpha^{-1}_{A,B,C}}
  \ar[dr]_{A\otimes\gamma_{B,C}}&&&(C\otimes A)\otimes B.\\
  &A\otimes(C\otimes B)\ar[r]_{\alpha^{-1}_{A,C,B}}&
  (A\otimes C)\otimes B\ar[ur]_{\gamma_{A,C}\otimes B}
}
\]
This implies compatibility with unitors, that is, a commuting
diagram
\[
\xymatrix@C-2em{
  A\otimes\Unit \ar[dr]_{\rho_A} \ar[rr]^{\gamma_{A,\Unit}}&&
  \Unit\otimes A \ar[dl]^{\lambda_A}\\
  &A.
}
\]
A \emph{symmetric monoidal category} is a braided monoidal
category that, in addition, satisfies
\(\gamma_{A,B}\gamma_{B,A} = \Id_{A\otimes B}\) for all objects
\(A\) and~\(B\).

An additive (braided) monoidal category is a category that is
at the same time additive and (braided) monoidal, and such that
the bifunctor~\(\otimes\) is additive.  We will only consider
additive monoidal categories.

\begin{example}
  \label{exa:monoidal_Abelian}
  The basic example of an additive symmetric monoidal category is
  the category of Abelian groups with the usual tensor product,
  \(\Unit=\Z\), and the obvious associator, unitors, and brading.
\end{example}

\begin{example}
  \label{exa:monoidal_Born}
  Let~\(\Born\) be the category of complete convex bornological
  vector spaces.  Let~\(\otimes\) be the complete projective
  bornological tensor product, usually denoted~\(\hot\), and
  let~\(\Unit\) be~\(\C\), assuming we are dealing with complex
  vector spaces.  Then the obvious associator, unitors, and
  braiding on the algebraic tensor product induce maps on the
  completions and provide the corresponding data in~\(\Born\).
  This turns~\(\Born\) into a symmetric monoidal category.

  Our examples will all be in this symmetric monoidal category.
\end{example}

\begin{example}
  \label{exa:monoidal_Indban}
  Let~\(\Indban\) be the category of inductive systems of
  Banach spaces.  The projective Banach space tensor product
  has a unique extension~\(\otimes\) to~\(\Indban\) that
  commutes with inductive limits.  Let~\(\Unit\) be~\(\C\),
  assuming we are dealing with complex vector spaces.  There
  are an obvious associator, unitors, and braiding that turn
  this into a symmetric monoidal category.  We refer
  to~\cite{Meyer:HLHA} for more details and an explanation why
  it is useful to replace~\(\Born\) by~\(\Indban\).
\end{example}

\begin{example}
  \label{exa:monoidal_Top}
  Let~\(\Top\) be the category of complete locally convex
  topological vector spaces.  Let~\(\otimes\) be the complete
  projective topological tensor product, usually
  denoted~\(\prot\), and let~\(\Unit\) be~\(\C\), assuming we
  are dealing with complex vector spaces.  Then the obvious
  associators, unitors, and braidings on the algebraic tensor
  products extend to the completions and provide the
  corresponding data in~\(\Top\).  This turns~\(\Top\) into a
  symmetric monoidal category.
\end{example}

\begin{definition}
  \label{def:closed_smc}
  A monoidal category~\(\Cat\) is called (left) \emph{closed}
  if the tensor product functor \(B\mapsto A\otimes B\) has a
  right adjoint for each object~\(A\).  In this case, the
  adjoints define a bifunctor \(\Cat^\op\times\Cat\to\Cat\),
  \((A,B)\mapsto \ihom AB\) with natural isomorphisms
  \begin{equation}
    \label{eq:define_ihom}
    \Cat(A\otimes B, C) \cong \Cat(B, \ihom AC)  
  \end{equation}
  for all objects \(A\), \(B\) and~\(C\) of~\(\Cat\).  The
  isomorphisms in~\eqref{eq:define_ihom} provide natural
  transformations \(\ev_{AB}\colon A\otimes (\ihom AB) \to B\),
  called \emph{evaluation map}, and \(B\to \ihom A{(A\otimes
    B)}\).
\end{definition}

\begin{example}
  \label{exa:monoidal_Born_closed}
  The symmetric monoidal category~\(\Born\) in
  Example~\ref{exa:monoidal_Born} is closed.  The internal Hom
  space \(\ihom AC\) is the space of bounded linear maps \(A\to
  C\) equipped with the bornology of equibounded sets of linear
  maps.  This bornology is complete and convex if~\(C\) is, and
  the defining isomorphism~\eqref{eq:define_ihom} is
  well-known.

  Banach spaces form an additive subcategory of~\(\Born\) that
  is closed both under \(\otimes\) and~\(\ihom{}{}\).  Hence
  they form a closed symmetric monoidal category in their own
  right.  The category of inductive systems \(\Indban\) is
  closed symmetric monoidal as well, see~\cite{Meyer:HLHA} for
  the construction of the internal Hom functor in \(\Indban\).
\end{example}

\begin{example}
  \label{exa:monoidal_Top_closed}
  The symmetric monoidal category~\(\Top\) in
  Example~\ref{exa:monoidal_Top} is not closed.  The complete
  projective topological tensor product functor cannot have a right
  adjoint because this would force it to commute with arbitrary
  colimits.  But it does not even commute with direct sums.

  The complete inductive tensor product
  of~\cite{Grothendieck:Produits} does commute with direct
  sums.  It is defined by a universal property for separately
  continuous bilinear maps.  But since separately continuous
  bilinear maps need not extend from dense subspaces, it seems
  likely that the \emph{completed} inductive tensor product is
  not associative in complete generality.  It is, therefore,
  unclear how to turn the category of all complete locally
  convex topological vector spaces into a closed monoidal category.
\end{example}

The internal Hom functor of a closed monoidal category comes with
several canonical maps (see also~\cite{Saavedra:Tannakiennes}).  The
most important ones are:
\begin{itemize}
\item a lifting of the adjointness
  isomorphism~\eqref{eq:define_ihom} to internal Homs:
  \begin{equation}
    \label{eq:define_ihom_lift}
    \ihom{(A\otimes B)}C \cong \ihom B{(\ihom AC)};    
  \end{equation}

\item the canonical \emph{composition map}
  \[
  (\ihom XY)\otimes (\ihom YZ) \to \ihom XZ,\qquad
  f\otimes g\mapsto g\circ f,
  \]
  for three objects \(X\), \(Y\) and~\(Z\), which is adjoint to
  the composition
  \[
  X\otimes\bigl((\ihom XY) \otimes (\ihom YZ)\bigr)
  \cong
  \bigl(X\otimes(\ihom XY)\bigr) \otimes (\ihom YZ) \to
  Y \otimes (\ihom YZ) \cong Z;
  \]
\item and the \emph{inflation map}
  \[
  \ihom XY \to \ihom{(Z\otimes X)}{(Z\otimes Y)},\qquad
  f\mapsto Z\otimes f = \Id_Z \otimes f,
  \]
  for three objects \(X\), \(Y\), and~\(Z\), which is adjoint
  to the map \((Z\otimes X)\otimes (\ihom XY) \cong Z\otimes
  \bigl(X\otimes (\ihom XY)\bigr) \to Z\otimes Y\).
\end{itemize}

Let~\(\Cat\) be an additive monoidal category.  An \emph{algebra}
in~\(\Cat\) is simply a semigroup object in~\(\Cat\), that is, an
object~\(A\) with a map \(\mu\colon A\otimes A\to A\) called
\emph{multiplication map}, such that the usual associativity diagram
\[
\xymatrix{
  (A\otimes A)\otimes A \ar[r]^{\alpha}\ar[d]_{\mu\otimes A}&
  A\otimes (A\otimes A) \ar[r]^-{\mu}&
  A\otimes A \ar[d]^{\mu}\\
  A\otimes A \ar[rr]_{\mu}&&
  A
}
\]
commutes.  A \emph{unital algebra} in~\(\Cat\) is a monoid
object in~\(\Cat\), that is, it is an algebra together with a
morphism \(\eta\colon \Unit\to A\) called \emph{unit} such that
the diagram
\[
\xymatrix{
  \Unit\otimes A \ar[r]^{\eta\otimes A} \ar[dr]_{\lambda_A}&
  A\otimes A \ar[d]^{\mu}&
  A\otimes\Unit \ar[l]_{A\otimes\eta} \ar[dl]^{\rho_A}\\
  &A
}
\]
commutes.  The usual trick shows that if an algebra has a left
and a right unit, then both coincide and provide a two-sided
unit.  In particular, the unit of an algebra is unique if it
exists.

Let \((A,\mu)\) be an algebra in~\(\Cat\).  A \emph{left
  \(A\)\nb-module} is an object~\(X\) of~\(\Cat\) with a map
\(\mu_X\colon A\otimes X\to X\), also called \emph{multiplication},
such that the usual associativity diagram commutes:
\[
\xymatrix{
  (A\otimes A)\otimes X \ar[r]^{\alpha}\ar[d]_{\mu\otimes A}&
  A\otimes (A\otimes X) \ar[r]^-{\mu_X}&
  A\otimes X \ar[d]^{\mu_X}\\
  A\otimes X \ar[rr]_{\mu}&&
  X
}
\]
A module over a unital algebra is \emph{unital} if the following
diagram commutes:
\[
\xymatrix{
  \Unit\otimes X \ar[r]^{\eta\otimes X} \ar[dr]_{\lambda_X}&
  A\otimes X \ar[d]^{\mu_X}\\
  &X
}
\]

Right modules and bimodules are defined similarly.  In a braided
monoidal category, any algebra~\(A\) has an opposite algebra~\(A^\op\)
with multiplication
\[
A\otimes A \xrightarrow{\gamma_{A,A}} A\otimes A
\xrightarrow{\mu} A,
\]
where \(\mu\) and~\(\gamma\) are the multiplication of~\(A\)
and the braiding.  The assumptions of a braided monoidal
category imply that this is again an algebra.

If \(\mu_X\colon X\otimes A\to X\) is a right \(A\)\nb-module
structure, then
\[
A\otimes X\xrightarrow{\gamma_{A,X}} X\otimes A
\xrightarrow{\mu_X} X
\]
is a left \(A^\op\)\nb-module structure on~\(X\), and vice
versa; once again we need the assumptions of a braided monoidal
category here.  Hence right \(A\)\nb-modules are equivalent to
left \(A^\op\)\nb-modules.  Thus there is no significant
difference between left and right modules in braided
monoidal categories.

\begin{example}
  \label{exa:algebras_modules}
  (Unital) algebras in the symmetric monoidal category of
  Abelian groups are (unital) rings, and (unital) modules over
  such algebras are (unital) modules over rings in the usual
  sense.

  (Unital) algebras in~\(\Top\) (Example~\ref{exa:monoidal_Top}) are
  complete locally convex topological (unital) algebras with jointly
  continuous multiplication \(A\times A\to A\).  (Unital) modules over
  them are complete locally convex topological (unital) modules with
  jointly continuous multiplication map \(A\times X\to X\).

  Similarly, (unital) algebras in~\(\Born\)
  (Example~\ref{exa:monoidal_Born}) are (unital) complete convex
  bornological algebras, and modules also have their usual
  meaning.
\end{example}

Next we define \(X\otimes_A Y\) and \(\ihom[A]XY\) for an
algebra~\(A\) and \(A\)\nb-modules \(X\) and~\(Y\).

\begin{definition}
  \label{def:balanced_tensor}
  Let~\(\Cat\) be a monoidal category in which each morphism has a
  cokernel.  Let~\(A\) be an algebra in~\(\Cat\), let~\(X\) be a right
  \(A\)\nb-module, and let~\(Y\) be a left \(A\)\nb-module, with
  multiplication maps \(\mu_X\colon X\otimes A\to X\) and
  \(\mu_Y\colon A\otimes Y\to Y\).  We define the \emph{balanced
    tensor product} \(X\otimes_A Y\) to be the cokernel of the map
  \[
  \mu_X\otimes Y - X\otimes\mu_Y\colon X\otimes A\otimes Y \to
  X\otimes Y.
  \]
  Roughly speaking, \(\mu_X\otimes Y - X\otimes\mu_A\) corresponds to
  the formula \(x\cdot a\otimes y- x\otimes a\cdot y\).
\end{definition}

\begin{definition}
  \label{def:balanced_Hom}
  Let~\(\Cat\) be a closed monoidal category in which each morphism
  has a kernel.  Let~\(A\) be an algebra in~\(\Cat\), and let \(X\)
  and~\(Y\) be left \(A\)\nb-modules with multiplication maps
  \(\mu_X\colon A\otimes X\to X\) and \(\mu_Y\colon A\otimes Y\to Y\).
  We define the \emph{balanced internal Hom} \(\ihom[A]XY\) to be
  the kernel of the map
  \[
  \ihom XY \to \ihom{(A\otimes X)}Y,\qquad
  f\mapsto f\circ\mu_X- \mu_Y\circ (A\otimes f).
  \]
  Roughly speaking, this corresponds to the map \(a\otimes x \mapsto
  f(a\cdot x) - a\cdot f(x)\).
\end{definition}

\begin{example}[\cite{Meyer:HLHA}]
  \label{exa:balanced_tensor_Born}
  If \(\Cat=\Born\), then \(X\otimes_A Y\) is the quotient of \(X\hot
  Y\) by the \emph{closed} linear span of \(xa\otimes y-x\otimes ay\)
  for \(x\in X\), \(a\in A\), \(y\in Y\).  Taking the closure ensures
  that the quotient is again separated, even complete.  And
  \(\ihom[A]XY\) is the space of bounded \(A\)\nb-module homomorphisms
  \(X\to Y\) with the equibounded bornology.
\end{example}

These constructions in general monoidal categories by and large have
the same formal properties as for rings and modules.  We will indicate
some of them now, see~\cite{Saavedra:Tannakiennes} for more details.

Let \(A\) and~\(B\) be algebras in~\(\Cat\), let~\(X\) be a
\(B,A\)-bimodule and~\(Y\) a left \(A\)\nb-module.  Assume that the
tensor product functor commutes with cokernels in both variables.
Then \(B\otimes (X\otimes_A Y)\) is the cokernel of the natural
map
\[
B\otimes \mu_X\otimes Y - B\otimes X\otimes \mu_Y\colon
B\otimes X\otimes A\otimes Y\to B\otimes X\otimes Y.
\]
Hence the multiplication map \(\mu_{BX}\otimes Y\colon B\otimes
X\otimes Y\to X\otimes Y\) descends to a map \(B\otimes (X\otimes_A Y)
\to X\otimes_A Y\), which turns \(X\otimes_A Y\) into a left
\(B\)\nb-module.  Similarly, an \(A,C\)\nb-module structure on~\(Y\)
induces a right \(C\)\nb-module structure on \(X\otimes_A Y\), and
if~\(X\) is a \(B,A\)-bimodule and~\(Y\) is an \(A,C\)-bimodule, then
\(X\otimes_A Y\) is a \(B,C\)-bimodule.

This allows us to form triple balanced tensor products \(X\otimes_A
Y\otimes_C Z\).  This is a \(B,D\)-bimodule if \(X\), \(Y\), and~\(Z\)
are bimodules over \(B,A\), \(A,C\), and \(C,D\) respectively.  More
precisely, we get two such bimodules, \((X\otimes_A Y)\otimes_C Z\)
and \(X\otimes_A (Y\otimes_C Z)\), which are related by a canonical
isomorphism that satisfies coherence laws similar to those
for~\(\otimes\).  Therefore, it is legitimate to drop brackets in such
tensor product expressions.

The internal Hom \(\ihom[A]XY\) inherits a \(B,C\)-bimodule
structure if~\(X\) is an \(A,B\)-bimodule and~\(Y\) is an
\(A,C\)-bimodule.  The \(B\)\nb-module structure in
\[
\Cat(B\otimes(\ihom[A]XY), \ihom[A]XY) \cong
\Cat_A(X\otimes B\otimes(\ihom[A]XY), Y)
\]
is adjoint to the \(A\)\nb-module map
\[
X\otimes B\otimes (\ihom[A]XY)
\xrightarrow{\mu_{XB}\otimes(\ihom[A]XY)}
X\otimes (\ihom[A]XY) \xrightarrow{\ev_{XY}}
Y,
\]
while the right \(C\)\nb-module structure in
\[
\Cat\bigl((\ihom[A]XY)\otimes C, \ihom[A]XY\bigr) \cong
\Cat_A(X\otimes(\ihom[A]XY)\otimes C, Y)
\]
is adjoint to the composition
\[
X\otimes (\ihom[A]XY)\otimes C
\xrightarrow{\ev_{XY}\otimes C}
Y\otimes C
\xrightarrow{\mu_{YC}}
Y.
\]
In examples, these definitions reproduce the usual bimodule structure
on spaces of linear maps, \(b\cdot f\cdot c(x) \defeq f(x\cdot b)\cdot
c\).  Routine diagram chases show that these maps define a
\(B,C\)\nb-bimodule structure.

The functors \(\otimes_A\) and \(\ihom[A]{}{}\) are related by
the expected adjointness relation:
\begin{equation}
  \label{eq:adjoint_tensor_hom}
  \Cat_{B,C}(X\otimes_A Y,Z) \cong \Cat_{A,C}(Y, \ihom[B]XZ),  
\end{equation}
where~\(X\) is a \(B,A\)-module, \(Y\) is an \(A,C\)-module, \(Z\) is
a \(B,C\)-module, and~\(\Cat_{B,C}\) denotes \(B,C\)\nb-module
homomorphisms.  Of course, we use the canonical bimodule structures on
\(X\otimes_A Y\) and \(\ihom[B]XZ\) here.  To
prove~\eqref{eq:adjoint_tensor_hom}, identify both sides with
subspaces of \(\Cat(X\otimes Y,Z) \cong \Cat(Y,\ihom XZ)\)
by~\eqref{eq:define_ihom} and check that the additional conditions
involving \(A\), \(B\) and~\(C\) correspond to each other.

\section{Self-induced algebras, smooth and rough modules}
\label{sec:self-induced}

From now on, we fix a closed monoidal category~\(\Cat\) with tensor
product functor~\(\otimes\), tensor unit~\(\Unit\), and internal Hom
functor~\(\ihom{}{}\).  We also assume that all morphisms in~\(\Cat\)
have a kernel and a cokernel, so that we may form the balanced tensor
product \(X\otimes_A Y\) and the balanced internal Hom \(\ihom[A]XY\).

Algebras and modules are understood to be algebras and modules
in~\(\Cat\).  Readers may think about the category
\(\Cat=\Born\) of complete convex bornological vector spaces
(Example~\ref{exa:monoidal_Born}), where our constructions
become much more concrete.  The functors \(X\otimes_A Y\) and
\(\ihom[A]XY\) in this case are described in
Example~\ref{exa:balanced_tensor_Born}.

Let~\(A\) be an algebra with multiplication \(\mu\colon
A\otimes A\to A\), and let~\(X\) be a left \(A\)\nb-module with
multiplication \(\mu_X\colon A\otimes X\to X\).  The
associativity relation
\[
\mu_X\circ (A\otimes\mu_X) = \mu_X\circ (\mu\otimes X)
\]
means that~\(\mu_X\) descends to a canonical map \(A\otimes_A X
\to X\) by the definition of the balanced tensor product.  We
denote this induced map by \(\bar\mu_X\colon A\otimes_A X\to
X\).  This is an \(A\)\nb-module homomorphism with respect to
the canonical \(A\)\nb-module structure on \(A\otimes_A X\)
by left multiplication.

In particular, for \(X=A\) the multiplication on~\(A\) induces
a map \(\bar\mu\colon A\otimes_A A \to A\).  This is an
\(A\)\nb-bimodule homomorphism with respect to the canonical
\(A\)\nb-bimodule structures on~\(A\) and \(A\otimes_A A\) from
left and right multiplication.

The associativity of~\(\mu_X\) also implies that~\(\bar\mu_X\)
is an \(A\)\nb-module homomorphism \(A\otimes_A X\to X\).
Hence the adjointness isomorphism
\[
\Cat_A(A\otimes_A X,X) \cong \Cat_A(X, \ihom[A]AX)
\]
in~\eqref{eq:adjoint_tensor_hom} turns~\(\bar\mu_X\) into an
\(A\)\nb-module homomorphism \(\bar\mu_X^\dagger\colon X \to
(\ihom[A] AX)\).

\begin{example}
  \label{exa:barmu_dagger}
  In the category of bornological vector spaces, \(\bar\mu_X\)
  is the map on the quotient \(A\otimes_A X\) of \(A\hot X\)
  induced by the map \(A\hot X\to X\), \(a\otimes x\mapsto
  a\cdot x\), and~\(\bar\mu_X^\dagger\) maps \(x\in X\) to the
  \(A\)\nb-module map \(A\to X\), \(a\mapsto a\cdot x\).
\end{example}

The natural maps \(\bar\mu\), \(\bar\mu_X\), and~\(\bar\mu_X^\dagger\)
are needed for our main definitions:

\begin{definition}
  \label{def:self-induced}
  An algebra~\(A\) is called \emph{self-induced}
  (see~\cite{Gronbaek:Morita_self-induced}) if the map
  \(\bar\mu\colon A\otimes_A A\to A\) is an isomorphism.
  Let~\(A\) be a self-induced algebra.

  A left \(A\)\nb-module~\(X\) is called \emph{smooth} if the
  map \(\bar\mu_X\colon A\otimes_A X\to X\) is an isomorphism, and
  \emph{rough} if the map \(\bar\mu_X^\dagger\colon X \to
  (\ihom[A]AX)\) is an isomorphism.
\end{definition}

We may define smoothness for right \(A\)\nb-modules by requiring the
analogous map \(X\otimes_A A\to X\) to be invertible.

To correctly define roughness for right modules, we must use
the \emph{right} internal Hom functor, that is, the right
adjoint to \(A\mapsto B\otimes A\).  In a braided monoidal
category, the right and left internal Hom functors are
naturally isomorphic; we may even view right \(A\)\nb-modules
as left \(A^\op\)\nb-modules, and~\(A\) is self-induced if and
only if~\(A^\op\) is self-induced.  Therefore, in the tensor
categories of greatest interest, there is no need for a
separate definition of roughness for right modules.  Since we
will not use rough right modules in the following, we do not
examine this technical issue any further here.

\begin{definition}
  \label{def:smoothening}
  Let~\(A\) be a self-induced algebra.  The \emph{smoothening} and
  \emph{roughening} of a left \(A\)\nb-module~\(X\) are defined by
  \[
  \Smooth_A(X) = \Smooth (X) \defeq A\otimes_A X,\qquad
  \Rough_A(X) = \Rough (X) \defeq \ihom[A]AX.
  \]
  This defines functors on the category of \(A\)\nb-modules.
  The maps \(\bar\mu_X\) and~\(\bar\mu^\dagger_X\) provide
  natural transformations \(\bar\mu_X\colon \Smooth(X)\to X\)
  and \(\bar\mu^\dagger_X\colon X\to\Rough(X)\).
\end{definition}

If~\(\Cat\) is the symmetric monoidal category of Banach spaces
with the projective Banach space tensor product, then the
self-induced algebras are exactly those of Niels
Gr{\o}nb{\ae}k~\cite{Gronbaek:Morita_self-induced}, and the
smooth modules are the \(A\)\nb-induced modules
of~\cite{Gronbaek:Morita_self-induced}.

Notice that we only defined smooth and rough modules and the
smoothening and roughening functors for a self-induced
algebra~\(A\).  If~\(A\) is self-induced, then~\(A\) is smooth
as an \(A\)\nb-bimodule.

Of course, self-induced algebras, smooth modules, and the smoothening
functor make sense without an internal Hom functor.  Thus we may still
speak of self-induced complete locally convex topological algebras,
smooth topological modules over them, and smoothenings of such modules
(Example~\ref{exa:monoidal_Top}).  But here we are mainly interested
in the interplay between smooth and rough modules.

\begin{example}
  \label{exa:smooth_rough_group}
  Let~\(\Cat\) be the category \(\Born\) of complete convex
  bornological vector spaces with the complete projective
  bornological tensor product
  (Example~\ref{exa:monoidal_Born}).  Fran\c{c}ois
  Bruhat~\cite{Bruhat:Distributions} used the
  Montgomery--Zippin structure theory for locally compact
  groups to define a space of smooth functions on~\(G\) for any
  locally compact group~\(G\).  For Lie groups, smoothness has
  the usual meaning, for totally disconnected groups, smooth
  functions are locally constant.  The smooth, compactly
  supported functions on~\(G\) form an algebra under
  convolution.  This is a complete convex bornological algebra
  \(\Ccinf(G)\), see~\cite{Meyer:Smooth} for the definition of
  the bornology.

  It is proved in~\cite{Meyer:Smooth} that the algebra
  \(\Ccinf(G)\) in~\(\Born\) is self-induced and that the
  category of smooth \(\Ccinf(G)\)-modules is isomorphic to the
  category of smooth representations of~\(G\) on complete
  convex bornological vector spaces.  Furthermore,
  \cite{Meyer:Smooth} introduces smoothening and roughening
  functors for \(\Ccinf(G)\).  These constructions
  in~\cite{Meyer:Smooth} are special cases of
  Definition~\ref{def:smoothening}.  In the following, we will
  generalise some of the results of~\cite{Meyer:Smooth} to
  arbitrary self-induced algebras.
\end{example}

\begin{proposition}
  \label{pro:unital_case}
  Let~\(A\) be a unital algebra.  Then~\(A\) is self-induced,
  and a left \(A\)\nb-module is smooth if and only if it is
  rough if and only if it is unital.

  Conversely, if~\(A\) is rough as a left \(A\)\nb-module,
  then~\(A\) is a unital algebra.
\end{proposition}

Since~\(A\) is always a smooth \(A\)\nb-module, it follows that
unital algebras are the only ones for which smooth and rough
modules are the same.

\begin{proof}
  Assume first that~\(A\) has a unit.  Then the
  \(A\)\nb-modules \(A\otimes X\) and \(\ihom AX\) are unital
  for all~\(X\).  Hence so are \(A\otimes_A X\) as a quotient
  of \(A\otimes X\), and \(\ihom[A]AX\) as a submodule of
  \(\ihom AX\).  Therefore, smooth modules and rough modules
  are unital.  Conversely, let~\(X\) be a unital left or right
  module.

  We define a map \(s_X\colon X\to A\otimes X\) by composing the
  unitor \(X\to \Unit\otimes X\) and the unit map \(\Unit\to A\)
  tensored with~\(X\).  The induced map \(X\to A\otimes_A X\) is a
  section for~\(\bar\mu_X\).  We also get a map \(s_{A\otimes X}\colon
  A\otimes X \to A\otimes A\otimes X\).  Let
  \[
  b'\defeq \mu\otimes X - A\otimes \mu_X\colon A\otimes A\otimes X\to A\otimes X
  \]
  be the map whose cokernel is the balanced tensor product
  \(A\otimes_A X\).  We compute \(b'\circ s_{A\otimes X} +
  s_X\circ \mu_X= \Id_{A\otimes X}\).  This implies
  that~\(\bar\mu_X\) is invertible.  Thus unital modules are
  smooth.  In particular, \(A\) is self-induced.

  We also define a map \(s'_X\colon (\ihom AX) \to X\) by
  composing with the unit map \(\Unit\to A\) and identifying
  \((\ihom\Unit X) \cong X\).  There is a similar map
  \[
  s''_X = s'_{\ihom AX}\colon \ihom {(A\otimes A)}X \cong \ihom
  A{(\ihom AX)} \to \ihom AX
  \]
  that composes with \(A\otimes\Unit\colon A\to A\otimes A\).
  We compute that~\(s'_X\) is a section for~\(\mu_X^\dagger\)
  and that \(s''_X \circ (b')^\dagger + \mu_X^\dagger\circ s'_X
  = \Id_{\ihom AX}\), where~\((b')^\dagger\) is the map whose
  kernel is \(\ihom[A]AX\) and \(\mu_X^\dagger\colon X\to \ihom
  AX\) is adjoint to~\(\mu_X\).  This implies
  that~\(\bar\mu_X^\dagger\) is invertible, that is, unital
  modules are rough.

  Now let~\(A\) be an arbitrary self-induced algebra and assume
  that~\(A\) is rough as a left \(A\)\nb-module.  That is, the
  canonical map \(A\to\ihom[A]AA\) is invertible.  Adjoint
  associativity yields \(\Cat(\Unit, \ihom[A]AA) \cong
  \Cat_A(A,A)\), and this always contains a canonical element:
  the identity map on~\(A\).  If the map \(A\to\ihom[A]AA\) is
  invertible, then we get a unique \(\eta\in \Cat(\Unit,A)\)
  with \(\mu\circ (A\otimes\eta)=\Id_A\), that is, \(\eta\) is
  a right unit element.  Consider the map
  \(\mu\circ(\eta\otimes A)\colon A\to A\).  When we compose it
  with the isomorphism \(A\to\ihom[A]AA\), we get again the
  canonical map \(A\to \ihom[A]AA\) because~\(\eta\) is a right
  unit.  Since the map \(A\to\ihom[A]AA\) is invertible, it
  follows that \(\mu\circ(\eta\otimes A)\) is equal to the
  identity map, that is, \(\eta\) is a left unit as well.
\end{proof}

For a unital algebra, any module~\(X\) decomposes naturally as
a direct sum \(X_0\oplus X_1\), where~\(X_0\) carries the zero
module structre and~\(X_1\) is a unital module.  Both the
smoothening and the roughening functors map~\(X\) to~\(X_1\),
and the natural maps \(\Smooth(X)\to X\to\Rough(X)\) are the
maps \(X_1\to X\to X_1\) from the direct sum decomposition
\(X\cong X_0\oplus X_1\).

The following proposition summarises the formal properties of the
smoothening and roughening functors:

\begin{theorem}
  \label{the:properties_smoothen_roughen}
  The following diagram commutes, and the indicated maps are
  isomorphisms:
  \[
  \xymatrix{
    \Smooth\Smooth(X) \ar[r]_{\cong}^{\bar\mu_{\Smooth(X)}} \ar[d]_{\Smooth(\bar\mu_X)}^\cong&
    \Smooth(X) \ar@{=}[dl] \ar[r]^{\bar\mu^\dagger_{\Smooth(X)}} \ar[d]_{\bar\mu_X}&
    \Rough\Smooth(X) \ar[d]^{\Rough(\bar\mu_X)}_\cong\\
    \Smooth(X) \ar[r]^{\bar\mu_X} \ar[d]^\cong_{\Smooth(\bar\mu^\dagger_X)}&
    X \ar[r]^{\bar\mu^\dagger_X} \ar[d]_{\bar\mu^\dagger_X}&
    \Rough(X)  \ar@{=}[dl] \ar[d]^{\Rough(\bar\mu^\dagger_X)}_\cong\\
    \Smooth\Rough(X) \ar[r]_{\bar\mu_{\Rough(X)}}&
    \Rough(X) \ar[r]_{\bar\mu^\dagger_{\Rough(X)}}^\cong&
    \Rough\Rough(X)
  }
  \]
  That is, the two canonical maps
  \(\Smooth\Smooth(X)\to\Smooth(X)\) and
  \(\Rough(X)\to\Rough\Rough(X)\) are equal, and there are
  natural isomorphisms \(\Smooth\Smooth \cong \Smooth \cong
  \Smooth\Rough\) and \(\Rough\Smooth \cong \Rough \cong
  \Rough\Rough\).  In particular, modules of the form
  \(\Smooth(X)\) are always smooth and modules of the form
  \(\Rough(X)\) are always rough.

  The functor~\(\Smooth\) is left adjoint to~\(\Rough\), that is, there
  is a natural isomorphism
  \[
  \Cat_A\bigl(\Smooth(X),Y\bigr) \cong \Cat_A\bigl(X,\Rough(Y)\bigr)
  \]
  for all \(A\)\nb-modules \(X\) and~\(Y\).

  The functor \(X\mapsto\Smooth(X)\) is right adjoint to the embedding
  of the category of smooth modules: composition with~\(\bar\mu_X\)
  induces an isomorphism
  \[
  \Cat_A\bigl(X,\Smooth(Y)\bigr) \cong \Cat_A(X,Y)
  \]
  if~\(X\) is a smooth \(A\)\nb-module and~\(Y\) is any
  \(A\)\nb-module.

  The functor \(X\mapsto\Rough(X)\) is left adjoint to the embedding
  of the category of rough modules: composition
  with~\(\bar\mu_X^\dagger\) induces an isomorphism
  \[
  \Cat_A(\Rough(X),Y) \cong \Cat_A(X,Y)
  \]
  if~\(X\) is any \(A\)\nb-module and~\(Y\) is a rough
  \(A\)\nb-module.
\end{theorem}

\begin{proof}
  Since~\(A\) is self-induced, we have \(A\otimes_A A \cong A\).
  Using the associativity of the balanced tensor product, this implies
  \[
  \Smooth\Smooth(X)
  \defeq A\otimes_A (A\otimes_A X)
  \cong (A\otimes_A A) \otimes_A X
  \cong A\otimes_A X
  \eqdef \Smooth(X).
  \]
  This isomorphism \(\Smooth\Smooth(X)\to\Smooth(X)\) is induced by
  \(\mu\otimes X\colon A\otimes A\otimes X\to A\otimes_A X\), that is,
  it is equal to \(\bar\mu_{\Smooth(X)}\); thus~\(\Smooth(X)\) is a
  smooth module.  Moreover, \(\mu\otimes
  X=A\otimes\mu_X\) as maps to \(A\otimes_A X\), so that
  \(\bar\mu_{\Smooth(X)} = \Smooth(\bar\mu_X)\).

  The adjointness of \(\Smooth\) and~\(\Rough\) is a special case of
  the adjointness between balanced tensor products and internal homs:
  \[
  \Cat_A(\Smooth(X), Y)
  \defeq \Cat_A(A\otimes_A X,Y)
  \cong \Cat_A(X,\ihom[A]AY)
  \eqdef \Cat_A\bigl(X, \Rough(Y)\bigr).
  \]
  The natural isomorphism \(\Smooth\circ\Smooth\cong\Smooth\) induces
  a natural isomorphism \(\Rough\cong\Rough\circ\Rough\) for the right
  adjoint functors.  A routine computation, which we omit,
  shows that this isomorphism is induced by
  \(\bar\mu_{\Rough(X)}^\dagger = \Rough(\bar\mu_X^\dagger)\).

  Next we show that composition with~\(\bar\mu_Y\) is an isomorphism
  \(\Cat_A\bigl(X,\Smooth(Y)\bigr) \cong \Cat_A(X,Y)\) if~\(X\) is a
  smooth \(A\)\nb-module and~\(Y\) is any \(A\)\nb-module.  We claim
  that its inverse is the composite
  \[
  \Cat_A(X,Y)
  \xrightarrow{\Smooth} \Cat_A\bigl(\Smooth(X),\Smooth(Y)\bigr)
  \xleftarrow[\cong]{\bar\mu_X^*} \Cat_A\bigl(X,\Smooth(Y)\bigr),
  \]
  where we use that~\(X\) is smooth, so that~\(\bar\mu_X^*\) is
  invertible.

  The naturality of the transformation~\(\bar\mu\) yields commuting
  diagrams
  \[
  \xymatrix{
    \Smooth(X) \ar[d]^\cong_{\bar\mu_X}\ar[r]^{\Smooth(f)}&\Smooth(Y)
    \ar[d]^{\bar\mu_Y}\\
    X \ar[r]_f&Y\\
  }
  \]
  for all \(f\in\Cat_A(X,Y)\).  Thus the composition
  \(\Cat_A(X,Y)\to\Cat_A(X,\Smooth(Y))\to \Cat_A(X,Y)\) is the
  identity map.  If
  \(f\in\Cat_A\bigl(X,\Smooth(Y)\bigr)\), then the diagram
  \[
  \xymatrix@C+1em{
    \Smooth(X) \ar@{=}[r] \ar[d]_{\Smooth(f)}&
    A\otimes_A X \ar[r]^-{\bar\mu_X}_-\cong \ar[d]_{A\otimes_A f} &
    X \ar[d]_{f} \ar[dr]^{\bar\mu_Y\circ f}\\
    \Smooth(\Smooth Y) \ar@{=}[r]&
    A\otimes_A A\otimes_A Y \ar[r]_-{\bar\mu\otimes_A Y}&
    A\otimes_A Y \ar[r]_-{\bar\mu_Y} & Y
  }
  \]
  commutes.  Since \(\bar\mu\otimes_A Y = A\otimes_A \bar\mu_Y\), we
  get \(f = \bigl(\Smooth (\bar\mu_Y
  f)\bigr)\circ(\bar\mu_X)^{-1}\).  This means that the
  composite map \(\Cat_A(X,\Smooth(Y))\to\Cat_A(X,Y)\to
  \Cat_A(X,\Smooth(Y))\) is the identity map as well.

  The adjointness relations already established imply
  \[
  \Cat_A\bigl(X,\Rough\circ\Smooth(Y)\bigr)
  \cong \Cat_A\bigl(\Smooth(X),\Smooth(Y)\bigr)
  \cong \Cat_A(\Smooth(X),Y)
  \cong \Cat_A\bigl(X,\Rough(Y)\bigr)
  \]
  for all \(A\)\nb-modules \(X\) and~\(Y\).  Hence
  \(\Rough\circ\Smooth(Y) \cong \Rough(Y)\) by the Yoneda Lemma.

  Let~\(Y\) be an \(A\)\nb-module and let~\(X\) be a smooth
  \(A\)\nb-module.  Then
  \[
  \Cat_A\bigl(X,\Smooth\Rough(Y)\bigr)
  \cong \Cat_A\bigl(X,\Rough(Y)\bigr)
  \cong \Cat_A(\Smooth(X),Y)
  \cong \Cat_A(X,Y)
  \cong \Cat_A\bigl(X,\Smooth(Y)\bigr).
  \]
  Hence the Yoneda Lemma yields \(\Smooth\circ\Rough(Y)\cong
  \Smooth(Y)\).

  It is routine to check that these isomorphisms
  \(\Rough\Smooth(Y)\cong\Rough(Y)\) and
  \(\Smooth(Y)\cong\Smooth\Rough(Y)\) are the canonical maps
  \(\Rough(\bar\mu_Y)\) and \(\Smooth(\bar\mu_Y^\dagger)\).

  If~\(Y\) is rough, that is, \(Y\cong\Rough(Y)\), then we
  compute
  \begin{multline*}
    \Cat_A(\Rough(X),Y)
    \cong \Cat_A\bigl(\Rough(X),\Rough(Y)\bigr)
    \cong \Cat_A(\Smooth\Rough(X),Y)
    \\\cong \Cat_A(\Smooth(X),Y)
    \cong \Cat_A\bigl(X,\Rough(Y)\bigr)
    \cong \Cat_A(X,Y),
  \end{multline*}
  that is, \(\Rough\) is left adjoint to the embedding of the
  category of rough modules.
\end{proof}

\section{Rough modules and unital modules over multiplier algebras}
\label{sec:modules_left_multiplier}

Let~\(A\) be a self-induced algebra in~\(\Cat\).  We view~\(A\)
as a left \(A\)\nb-module and let
\[
\Leftm(A)\defeq\ihom[A]AA
\]
be the algebra of left \(A\)\nb-module endomorphisms on~\(A\).
This is a unital algebra in~\(\Cat\).  It comes with a
canonical algebra homomorphism \(A\to\Leftm(A)\) by right
multiplication.

We may also view \(\Leftm(A)\) as the roughening of the left
\(A\)\nb-module structure on~\(A\), and the map \(A\to\Leftm(A)\) as
the canonical map \(\bar\mu^\dagger\colon A\to \Rough(A)\).
Theorem~\ref{the:properties_smoothen_roughen} implies \(A\otimes_A
\Leftm(A) \cong A\).  Roughly speaking, this means that~\(A\) is a
left ideal in \(\Leftm(A)\) (but the map \(A\to\Leftm(A)\) need not be
monic, see Section~\ref{sec:example_biprojective}).

If~\(X\) is a left \(A\)\nb-module, then the \(A\)\nb-module
structure on \(\Rough(X) \defeq \ihom[A]AX\) extends
canonically to a unital left \(\Leftm(A)\)-module structure
because~\(A\) is an \(A,\Leftm(A)\)-bimodule by construction.
Thus rough \(A\)\nb-modules become unital
\(\Leftm(A)\)-modules, and this provides a fully faithful
functor from the category of rough \(A\)\nb-modules to the
category of unital \(\Leftm(A)\)-modules.  Conversely, any
unital \(\Leftm(A)\)-module becomes an \(A\)\nb-module by
restricting the action.  But such restricted modules need not
be rough, and the restriction functor need not be fully
faithful.

To see this, consider free modules.  Free unital
\(\Leftm(A)\)-modules have the form \(\Leftm(A) \otimes V =
(\ihom[A]AA)\otimes V\) for some object~\(V\) of~\(\Cat\).  We
view this as a left \(A\)\nb-module and simplify its roughening
using Theorem~\ref{the:properties_smoothen_roughen} and the
associativity of~\(\otimes\):
\begin{multline*}
  \Rough(\Leftm(A)\otimes V)
  \cong \Rough\Smooth(\Leftm(A)\otimes V)
  \cong \Rough\bigl(\Smooth\bigl(\Leftm(A)\bigr)\otimes V\bigr)
  \\\cong \Rough\bigl(\Smooth \Rough(A)\otimes V\bigr)
  \cong \Rough(A\otimes V)
  = \ihom[A]A{(A\otimes V)}.
\end{multline*}
In general, \((\ihom[A]AA) \otimes V\) is different from
\(\ihom[A]A{(A\otimes V)}\).

We may also view smooth modules as modules over a suitable
\emph{right} multiplier algebra \(\Rightm(A)\).  This is a
unital algebra such that~\(A\) is an \(\Rightm(A),A\)-bimodule.
Since this involves a left module structure \(\Rightm(A)\otimes
A\to A\), we need the \emph{right} internal Hom functor defined
by \(\Cat(X\otimes Y,Z) \cong \Cat(X, \ihomr ZY)\).  This
functor is similar to \(\ihom XY\), but the evaluation maps are
of the form \(\ihomr XY\otimes X \to Y\), and the composition
maps are of the form \((\ihomr YZ) \otimes (\ihomr XY) \to
\ihomr XZ\).  Thus
\[
\Rightm(A) \defeq \ihomr[A]AA
\]
becomes a unital algebra such that~\(A\) is an
\(\Rightm(A),A\)-bimodule.  The \(A\)\nb-module structure on
\(\Smooth(X) \defeq A\otimes_A X\) extends canonically to a
left unital \(\Rightm(A)\)-module structure for any
\(A\)\nb-module~\(X\).  This provides a fully faithful
embedding from the category of smooth left \(A\)\nb-bimodules
to the category of unital left \(\Rightm(A)\)-modules.  Once
again, this functor is not an isomorphism of categories.

In general, \(\Leftm(A)\) and \(\Rightm(A)\) are different,
even in the symmetric monoidal category \(\Born\).  For
instance, this happens for the biprojective algebras \(V\otimes
W\) studied in Section~\ref{sec:example_biprojective}.

\section{Functoriality for homomorphisms and bimodules}
\label{sec:functorial}

Let \(A\) and~\(B\) be algebras in an additive closed monoidal
category.

\begin{definition}
  \label{def:module_category}
  Let \(\Mod{A}\) denote the category of smooth modules over a
  self-induced algebra~\(A\).
\end{definition}

An \(A,B\)-bimodule~\(M\) induces a functor \(M\otimes_B\blank\) from
the category of \(B\)\nb-modules to the category of \(A\)\nb-modules
and a functor \(\ihom[A]M\blank\) from the category of
\(A\)\nb-modules to the category of \(B\)\nb-modules.  These two
functors are adjoint to each other by~\eqref{eq:adjoint_tensor_hom}:
\begin{equation}
  \label{eq:adjoint_tensor_bimodule}
  \Cat_A(M\otimes_B X,Y)\cong \Cat_B(X,\ihom[A]MY)  
\end{equation}
if \(X\) and~\(Y\) are a \(B\)\nb-module and an \(A\)\nb-module,
respectively.  When do these functors preserve smooth or rough
modules?

The module \(\ihom[A]MY\) is usually not smooth, even if~\(Y\) is, and
\(M\otimes_B X\) is usually not rough, even if~\(X\) is.  But we have
the following positive results:

\begin{proposition}
  \label{pro:induce_smooth_rough}
  Let \(A\) and~\(B\) be algebras, assume that~\(A\) is self-induced.
  Let~\(Y\) be any left \(B\)\nb-module.  If~\(M\) is an
  \(A,B\)-bimodule that is smooth as a left \(A\)\nb-module, then
  \(M\otimes_B Y\) is a smooth \(A\)\nb-module.

  If~\(M\) is a \(B,A\)-bimodule that is smooth as a right
  \(A\)\nb-module, then \(\ihom[B]MY\) is a rough \(A\)\nb-module.
\end{proposition}

\begin{proof}
  The first assertion follows from the associativity of balanced
  tensor products:
  \[
  \Smooth(M\otimes_B Y)
  \defeq A\otimes_A (M\otimes_B Y)
  \cong (A\otimes_A M)\otimes_B Y
  \cong M\otimes_B Y.
  \]
  The second assertion uses a strengthening of the adjointness
  relation~\eqref{eq:adjoint_tensor_hom} as
  in~\eqref{eq:define_ihom_lift} with internal Hom functors
  instead of morphism sets.  Thus
  \[
  \Rough(\ihom[B]MY)
  \cong \ihom[A]A{(\ihom[B]MY)}
  \cong \ihom[B]{(M\otimes_A A)}Y
  \cong \ihom[B]MY.\qedhere
  \]
\end{proof}

For general~\(M\), we get smooth or rough modules if we compose
the two functors above with the smoothening or roughening
functors.  The functor \(\Smooth(\ihom[B]M\blank)\) maps
\(B\)\nb-modules to smooth \(A\)\nb-modules, and
\(\Rough(M\otimes_B\blank)\) maps \(B\)\nb-modules to rough
\(A\)\nb-modules.  The other two combinations of our functors
are not worth considering because the computations above show
\[
\Smooth_A(M\otimes_B X) \cong \Smooth_A(M) \otimes_B X,\qquad
\Rough_A(\ihom[B]MY) \cong \ihom[B]{(\Smooth_A M)}Y.
\]

\begin{proposition}
  \label{pro:adjointness_bimodule_functors}
  Let~\(M\) be a smooth \(A,B\)-bimodule.  Then the functors
  \begin{alignat*}{2}
    \Mod{B}&\to\Mod{A},&\qquad X&\mapsto M\otimes_B X,\\
    \Mod{A}&\to\Mod{B},&\qquad Y&\mapsto \Smooth_B(\ihom[A]MY),
  \end{alignat*}
  are adjoint to each other, that is, \(\Cat_A(M\otimes_B X,Y) \cong
  \Cat_B\bigl(X,\Smooth_B(\ihom[A]MY)\bigr)\) if~\(X\) is a smooth
  \(B\)\nb-module and~\(Y\) a smooth \(A\)\nb-module.
\end{proposition}

\begin{proof}
  Theorem~\ref{the:properties_smoothen_roughen} implies
  \(\Cat_B\bigl(X,\Smooth_B (\ihom[A]MY)\bigr) \cong
  \Cat_B(X,\ihom[A]MY)\), and this is isomorphic to
  \(\Cat_A(M\otimes_B X,Y)\) by~\eqref{eq:adjoint_tensor_bimodule}.
\end{proof}

We may define Morita equivalence for self-induced algebras as
in~\cite{Gronbaek:Morita_self-induced}:

\begin{definition}
  \label{def:Morita_equivalence}
  Two self-induced algebras \(A\) and~\(B\) are called
  \emph{Morita equivalent} if there exist a smooth
  \(A,B\)-bimodule~\(P\), a smooth \(B,A\)-bimodule~\(Q\), and
  natural isomorphisms \(P\otimes_B Q \cong A\) and
  \(Q\otimes_A P \cong B\).
\end{definition}

\begin{proposition}
  \label{pro:Morita_categories}
  If \(A\) and~\(B\) are Morita equivalent via the bimodules
  \(P\) and~\(Q\), then the categories of smooth \(A\)- and
  \(B\)\nb-modules are equivalent via the functors
  \begin{alignat*}{2}
    \Mod{A} &\to \Mod{B},&\qquad X&\mapsto Q\otimes_A X,\\
    \Mod{B} &\to \Mod{A},&\qquad Y&\mapsto P\otimes_B Y.
  \end{alignat*}
  The categories of rough \(A\)- and \(B\)\nb-modules are
  equivalent via the functors \(X\mapsto \ihom[A]PX\) and
  \(Y\mapsto \ihom[B]QY\).
\end{proposition}

\begin{proof}
  The equivalence \(\Mod{A}\cong\Mod{B}\) follows
  from the associativity of tensor products and the assumed
  isomorphisms \(P\otimes_B Q\cong A\), \(Q\otimes_A P\cong
  B\), and from the definition of smooth modules: \(A\otimes_A
  X\cong X\) and \(B\otimes_B Y\cong Y\).  The corresponding
  assertions about rough modules also use the adjointness
  relation~\eqref{eq:adjoint_tensor_bimodule}.
\end{proof}

Since the categories of smooth and rough modules are
equivalent, we may also construct the equivalence between rough
module categories from the equivalence between the smooth
module categories by first smoothening, then applying the
equivalence, and then roughening.  A straightforward
computation shows that this sequence of steps produces the
functor described above:
\begin{equation}
  \label{eq:rough_Morita_smooth}
  \Rough_A\bigl(P\otimes_B \Smooth_B(X)\bigr) \cong
  \ihom[B]QX
\end{equation}
for all rough \(B\)\nb-modules~\(X\).  First, since~\(P\) is
smooth, we compute
\[
P\otimes_B \Smooth_B(X)
\cong P\otimes_B (B\otimes_B X)
\cong (P\otimes_B B)\otimes_B X
\cong P\otimes_B X.
\]
The argument that shows that \(Q\otimes_A \blank\) is an equivalence
of categories shows more: tensoring with~\(Q\) induces an isomorphism
\[
\ihom[A] XY \cong \ihom[B]{(Q\otimes_A X)}{(Q\otimes_A Y)}
\]
for all smooth \(A\)\nb-modules \(X\) and~\(Y\).  Hence
\[
\ihom[A]A{(P\otimes_B X)}
\cong \ihom[B]{(Q\otimes_A A)}{\bigl(Q\otimes_A (P\otimes_B X)\bigr)}
\cong \ihom[B]Q{(B\otimes_B X)}
\cong \ihom[B]QX.
\]

Although the author is not aware of an example, it seems likely
that there exist equivalences \(\Mod{A} \cong \Mod{B}\) (even
for unital \(A\) and~\(B\)) that are not induced by a bimodule
as above.  To get a bimodule from an equivalence of categories,
we assume that further structure is preserved.  There is a
tensor product operation \(X,Y\mapsto X\otimes Y\) for a smooth
\(A\)\nb-module~\(X\) and an object~\(Y\) of~\(\Cat\), which
turns the category of smooth \(A\)\nb-modules into a right
\(\Cat\)\nb-category in the notation
of~\cite{Pareigis:C-categories}.  The functor
\(M\otimes_B\blank\) for a bimodule~\(M\) is a
\(\Cat\)\nb-functor in the notation
of~\cite{Pareigis:C-categories}, that is, there are natural
isomorphisms \(M\otimes_B (X\otimes Y) \cong (M\otimes_B X)\otimes Y\)
satisfying appropriate coherence laws.

\begin{proposition}
  \label{pro:functor_bimodule}
  Let \(A\) and~\(B\) be self-induced algebras.  A functor
  \[
  F\colon \Mod{A}\to\Mod{B}
  \]
  is of the form \(M\otimes_B\blank\) for a smooth
  \(B,A\)-bimodule~\(M\) if and only if it preserves cokernels
  and is a \(\Cat\)\nb-functor.  The bimodule~\(M\) is
  determined uniquely up to isomorphism.
\end{proposition}

\begin{proof}
  The underlying left \(B\)\nb-module of~\(M\) must be \(F(A)\), of
  course.  To get the right \(A\)\nb-module structure on \(M\defeq
  F(A)\), we use the multiplication map \(\mu\colon A\otimes A\to A\).
  This module homomorphism induces a natural \(B\)\nb-module map
  \[
  M\otimes A \defeq F(A)\otimes A \cong F(A\otimes A) \to F(A)
  \eqdef M.
  \]
  Now let~\(X\) be any smooth \(A\)\nb-module.  Then \(X\cong
  A\otimes_A X\) is the cokernel of the canonical map \(A\otimes
  A\otimes X\to A\otimes X\) that defines \(A\otimes_A X\).
  Since~\(F\) is compatible with cokernels and tensor products,
  \(F(X)\) is naturally isomorphic to the cokernel of an induced map
  \(F(A)\otimes A\otimes X\to F(A)\otimes X\).  But this is exactly
  the map that defines \(F(A)\otimes_A X\), so that \(F(X) \cong
  F(A)\otimes_A X\) for all smooth \(A\)\nb-modules~\(X\).  In
  particular, \(F(A)\) is smooth as a right \(A\)\nb-module.
  It is easy to see that \(F(A)\) with the bimodule structure
  described above is the only one that may induce the functor~\(F\).
\end{proof}

We may
use an algebra homomorphism \(f\colon A\to B\) to turn
\(B\)\nb-modules into \(A\)\nb-modules.  But when does this
functor~\(f^*\) preserve smoothness or roughness of bimodules?  To
analyse this, we use~\(f\) to view~\(B\) as an \(A,B\)-bimodule or as
a \(B,A\)-bimodule.  Then \(B\otimes_B X\cong X\) for smooth
\(B\)\nb-modules~\(X\), so that~\(f^*\) on smooth modules is the
tensor product functor for the \(A,B\)-bimodule~\(B\).  By
Proposition~\ref{pro:induce_smooth_rough}, this maps smooth
\(B\)\nb-modules to smooth \(A\)\nb-modules provided~\(B\) is smooth
as a left \(A\)\nb-module.  And \(X\cong (\ihom[B]BX)\) for a rough
\(B\)\nb-module~\(X\), so that~\(f^*\) on rough modules is the
internal Hom functor for the \(B,A\)-bimodule~\(B\).  By
Proposition~\ref{pro:induce_smooth_rough}, this maps rough
\(B\)\nb-modules to rough \(A\)\nb-modules provided~\(B\) is smooth as
a right \(A\)\nb-module.  Summing up:

\begin{lemma}
  \label{lem:pull-back_smooth_rough}
  Let \(f\colon A\to B\) be an algebra homomorphism.  Assume that~\(B\)
  is smooth both as a left and as a right \(A\)\nb-module.  Then the
  induced functor~\(f^*\) from \(B\)\nb-modules to \(A\)\nb-modules
  maps smooth modules to smooth modules and rough modules to rough
  modules.
\end{lemma}

More generally, the above construction only used compatible \(A,B\)-
and \(B,A\)-bimodule structures on~\(B\).  These still exist if we
replace~\(f\) by an algebra homomorphism into the multiplier algebra
(also called double centraliser algebra) of~\(B\).

Even if~\(B\) is not smooth as a left or right \(A\)\nb-module, the
above discussion shows how to get functors between the smooth and
rough module categories: simply replace the \(A,B\)- or
\(B,A\)-bimodule~\(B\) by the appropriate smoothening and argue
exactly as above.  Furthermore, we may turn the functor~\(f^*\) on
rough modules into one on smooth modules by composing with the
smoothening:
\[
X\mapsto \Smooth_A(B\otimes_B X) \cong \Smooth_A(f^*X),\qquad
X\mapsto \Smooth_A(\ihom[B]BX) \cong \Smooth_A(f^*\Rough_B X).
\]
for a smooth left \(B\)\nb-module~\(X\).

As a consequence, any algebra homomorphism from~\(A\) to the
multiplier algebra of~\(B\) yields two pairs of adjoint functors
between the categories of smooth modules over \(A\) and~\(B\).  The
first pair consists of the functors
\begin{alignat*}{2}
  \Mod{B}&\to\Mod{A},&\qquad X&\mapsto \Smooth_A(f^*X),\\
  \Mod{A}&\to\Mod{B},&\qquad Y&\mapsto \Smooth_B(\ihom[A]{(A\otimes_A B)}Y),
\end{alignat*}
the second pair of the functors
\begin{alignat*}{2}
  \Mod{A}&\to\Mod{B},&\qquad Y&\mapsto B\otimes_A Y,\\
  \Mod{B}&\to\Mod{A},&\qquad X&\mapsto \Smooth_A(\ihom[B]{(B\otimes_A A)}X).
\end{alignat*}

\section{Applications}
\label{sec:applications}

\subsection{A simple biprojective example}
\label{sec:example_biprojective}

First we consider a very simple and well-known class of
examples.  Let \(V\) and~\(W\) be objects of~\(\Cat\) and let
\(b\colon W\otimes V\to\Unit\) be a map.  Then \(A\defeq
W\otimes V\) becomes an associative non-unital algebra for the
product
\[
V\otimes W\otimes V\otimes W
\xrightarrow{V\otimes b\otimes W} V\otimes W.
\]
Similar maps define a left \(A\)\nb-module structure on~\(V\)
and a right \(A\)\nb-module structure on~\(W\).  We may also
view \(V\) and~\(W\) as an \(A,\Unit\)\nb-bimodule and a
\(\Unit,A\)\nb-bimodule because any object of~\(\Cat\) carries
a canonical unital \(\Unit\)\nb-bimodule structure given by the
left and right unitors.

From now on, we assume also that~\(b\) is non-degenerate in the
sense that there exist maps \(v\colon \Unit\to V\) and
\(w\colon \Unit\to W\) for which \(b\circ (w\otimes v)\) is the
canonical isomorphism \(\Unit\otimes\Unit\to\Unit\).  Then the
map \(V\otimes w\otimes v\otimes W\colon A\to A\otimes A\) is a
bimodule section for the multiplication map \(A\otimes A \to
A\).  This implies that~\(A\) is biprojective, that is, \(A\)
is projective as an \(A\)\nb-bimodule.  A straightforward
computation, which we omit, shows that~\(A\) is self-induced.
The bimodules \(V\) and~\(W\) are smooth and implement a Morita
equivalence between \(\Unit\) and~\(A\).

We may use this to describe the categories of smooth and rough
\(A\)\nb-modules.  First, Proposition~\ref{pro:unital_case}
identifies smooth and rough \(\Unit\)\nb-modules with unital
\(\Unit\)\nb-modules.  Since any object of~\(\Cat\) carries a
unique unital \(\Unit\)\nb-module structure, it follows that
the categories of smooth and rough \(\Unit\)\nb-modules are
both equivalent to~\(\Cat\).  Due to the Morita equivalence,
the categories of smooth and rough \(A\)\nb-modules are
equivalent to~\(\Cat\) as well
(Proposition~\ref{pro:Morita_categories}).  More precisely, the
equivalences map an object~\(X\) of~\(\Cat\) to \(V\otimes X\)
and \(\ihom WX\), respectively, where we use the left
\(A\)\nb-module structure on~\(V\) and the right
\(B\)\nb-module structure on~\(W\).

For instance, if~\(\Cat\) is the category of complete convex
bornological vector spaces (Example~\ref{exa:monoidal_Born}),
then we may take \(V=W=\bigoplus_\N \C\) with the obvious
pairing \(b(x,y) \defeq \sum_{n\in\N} x_ny_n\).  Then~\(A\)
is the algebra~\(\Mat_\infty\) of finite matrices.  Our
results show that the categories of smooth and rough
\(\Mat_\infty\)\nb-modules are both equivalent to the category of
complete convex bornological vector spaces, where a
bornological vector space~\(X\) corresponds to the smooth
\(\Mat_\infty\)\nb-module \(V\otimes X \cong \bigoplus_\N X\) and the rough
\(\Mat_\infty\)\nb-module \(\ihom WX \cong \prod_\N X\), with finite
matrices acting by the usual matrix--vector multiplication.

We may also take \(V_\Sch=W_\Sch=\Sch(\N)\) with the same formula
for~\(b\).  The resulting algebra is~\(\Rdk\), the algebra of rapidly
decreasing matrices.  Once again, we get a complete description of the
categories of smooth and rough \(\Rdk\)\nb-modules.  This time, the
tensor product \(V_\Sch\otimes X\) and the space \(\ihom{W_\Sch}X\) are
spaces of sequences in~\(X\) with certain growth conditions: \(V_\Sch
\otimes X\) consists of \emph{rapidly decreasing} sequences,
\(\ihom{W_\Sch}X\) of sequences of \emph{polynomial growth}.

A sequence~\((x_n)\) has rapid decay if there are a sequence of
scalars~\((\varepsilon_n)\) with rapid decay and a bounded subset
\(S\subseteq X\) with \(x_n\in \varepsilon_n \cdot S\) for all
\(n\in\N\).  A subset~\(T\) of \(V_\Sch\otimes X\) is bounded if it
has uniformly rapid decay: the same~\(\varepsilon_n\) and~\(S\)
work for all sequences in~\(T\).  This bornological vector
space of rapidly decreasing sequences is isomorphic to \(V_\Sch
\otimes X\).

A sequence~\((x_n)\) has polynomial growth if \(\{\varepsilon_n\cdot
x_n\mid n\in\N\}\) is bounded for each rapidly decreasing sequence of
scalars~\((\varepsilon_n)\).  A set~\(T\) of polynomial growth
sequences is bounded if it has uniform polynomial growth: the set
\(\bigcup_{(x_n)\in T} \{\varepsilon_n\cdot x_n\mid n\in\N\}\) is
bounded.  This bornological vector space of polynomial growth
sequences is isomorphic to \(\ihom{W_\Sch}X\).

The categories of \emph{all} \(\Mat_\infty\)\nb-modules and of
all \(\Rdk\)\nb-modules are not equivalent to the category of
complete convex bornological vector spaces: smooth and rough
modules are different for \(\Mat_\infty\) and~\(\Rdk\), while
they are the same for~\(\C\).

The left and right multiplier algebras of~\(A\) are the
roughenings of the canonical left and right module structures
on~\(A\).  Using the Morita equivalence to~\(\C\), we get
\[
\Rightm(A) \cong (\ihom VV),\qquad
\Leftm(A) \cong (\ihom WW).
\]
These obviously act on \(A\defeq V\otimes W\) on the left and
right by multiplication.  If we let~\(V\) be finite-dimensional
and~\(W\) infinite-dimensional, then the two algebras are
obviously quite different.

The \emph{multiplier algebra} (or double centraliser algebra) in this
case is
\[
\{(L,R) \in (\ihom VV)\times (\ihom WW) \mid
b\circ (\Id_W\otimes L) = b\circ (R\otimes \Id_V)\},
\]
where the multiplication uses the opposite multiplication on
\(\ihom WW\).

The natural map \(\Smooth(X) \to X\) for an \(A\)\nb-module~\(X\) need
not be monic (injective) for the algebras considered above.  Thus it
is necessary to assume approximate identities in
\cite{Meyer:Smooth}*{Lemma 4.4} even if the algebra in question is
self-induced.  For instance, take~\(X\) to be the right multiplier
algebra \(\Leftm(A) \cong \ihom WW\).  Since \(\Leftm(A) \cong
\Rough(A)\), we have \(\Smooth \Leftm(A) \cong A \defeq V\otimes W\), so that
we are dealing with the question whether the map \(V\otimes W\to\ihom
WW\) induced by~\(b\) is monic.  This fails if \(b\circ (\Id_V\otimes
f)=0\) for some map \(f\colon W_0\to W\), which is still allowed by
our rather weak non-degeneracy assumption on~\(b\).  Even if~\(b\) is
non-degenerate, say, if we work in the category of Banach spaces and
\(V=W^*\) is the dual Banach space of~\(W\), the map \(V\otimes W\to
\ihom WW\) may fail to be injective: this is related to a failure of
Grothendieck's Approximation Property for~\(W\).

\subsection{Lie group and Lie algebra representations}
\label{sec:Lie_group_algebra}

Let~\(\Cat\) be the tensor category of complete convex
bornological vector spaces.  Let~\(G\) be a connected Lie group with Lie
algebra~\(\g\).  Let \(\Ccinf(G)\) be the space of smooth,
compactly supported functions on~\(G\) with the convolution
product and the natural bornology, where a subset is bounded if
its functions are all supported in the same compact subset and
have uniformly bounded derivatives of all orders.  This is a
complete convex bornological algebra.  It is shown
in~\cite{Meyer:Smooth} that the category of smooth
group representations of~\(G\) on bornological vector spaces is
equivalent to the category of smooth \(\Ccinf(G)\)-modules
in~\(\Cat\).

Let~\(\Univ(\g)\) be the universal enveloping algebra of~\(\g\),
equipped with the fine bornology.  The category of bounded
Lie algebra representations of~\(\g\) on complete convex
bornological vector spaces is equivalent to the category of
unital \(\Univ(\g)\)-modules in~\(\Cat\).

A smooth group representation of~\(G\) may be differentiated to
a Lie algebra representation of~\(\g\), that is, to a unital
\(\Univ(\g)\)-module structure.  This provides a functor
\[
\Mod{\Ccinf(G)} \to \Mod{\Univ(\g)},
\]
the differentiation functor.  It is fully faithful if and only
if~\(G\) is connected.

The left regular representation of~\(G\) on~\(\Ccinf(G)\)
yields a Lie algebra representation of~\(\g\) on~\(\Ccinf(G)\).
If~\(V\) is a smooth representation of~\(G\) or, equivalently,
a smooth \(\Ccinf(G)\)-module, then the induced
\(\Univ(\g)\)-module structure on~\(V\) is the natural module
structure on \(V\cong \Ccinf(G)\otimes_{\Ccinf(G)} V\) induced
by the \(\Univ(\g)\)-module structure on~\(\Ccinf(G)\).  As a
consequence, the differentiation functor
\[
\diff\colon \Mod{\Ccinf(G)} \to \Mod{\Univ(\g)}
\]
is naturally isomorphic to the tensor product functor
\(V\mapsto \Ccinf(G)\otimes_{\Ccinf(G)} V\) with the canonical
\(\Univ(\g),\Ccinf(G)\)-bimodule structure on \(\Ccinf(G)\).

More explicitly, the representation of~\(\g\) on \(\Ccinf(G)\)
identifies~\(\g\) with the space of right-invariant vector
fields on~\(G\) and lets the latter act on \(\Ccinf(G)\) as
derivations.  The induced action of \(\Univ(\g)\) proceeds by
identifying \(\Univ(\g)\) with the algebra of right-invariant
differential operators on~\(G\).  Equivalently, we may identify
\(\Univ(\g)\) with the algebra of distributions on~\(G\) supported
at the identity element.  Since compactly supported
distributions on~\(G\) act on smooth functions by convolution
on the left and right, this provides left and right
\(\Univ(\g)\)-module structures on \(\Ccinf(G)\).  These commute
with the \(\Univ(\g)\)- and \(\Ccinf(G)\)-module structures on the
other side because convolution is associative.

By our general theory, the differentiation functor comes
together with three other functors.  First, it has a right
adjoint functor
\begin{multline*}
\diff^*\colon \Mod{\Univ(\g)} \to \Mod{\Ccinf(G)},\\
W\mapsto \Smooth_{\Ccinf(G)} \bigl(\ihom[\Univ(\g)]{\Ccinf(G)}W\bigr)
= \Smooth_{\Ccinf(G)} \Hom_{\Univ(\g)}(\Ccinf(G),W).
\end{multline*}
Since~\(G\) is connected, the differentiation functor is fully
faithful.  Equivalently, \(\diff^*\circ\diff(V) \cong V\) for
any smooth group representation~\(V\) of~\(G\).

Secondly, we may map smooth \(\Cinf(G)\)-modules to rough
\(\Cinf(G)\)-modules by the roughening functor, and then equip the
latter with a canonical \(\Univ(\g)\)-module structure --~rough
modules are sufficiently differentiable for such a \(\Univ(\g)\)-module
structure to exist.  We may rewrite this alternative differentiation
functor as
\begin{multline*}
  \bar\diff\colon \Mod{\Ccinf(G)} \to \Mod{\Univ(\g)},\\
  V\mapsto \Rough(W) = \Hom_{\Ccinf(G)}(\Ccinf(G),W)
  = \bigl(\ihom[\Ccinf(G)]{\Ccinf(G)}W\bigr),
\end{multline*}
where we view~\(\Ccinf(G)\) as a
\(\Ccinf(G),\Univ(\g)\)-bimodule by letting~\(\Univ(\g)\) act
by right convolution.  Finally, \(\bar\diff\) has a left
adjoint functor
\[
\bar\diff^*\colon \Mod{\Univ(\g)} \to \Mod{\Ccinf(G)},\qquad
W\mapsto \Ccinf(G) \otimes_{\Univ(\g)} W.
\]
Since~\(G\) is connected and roughening is fully faithful, the
functor~\(\bar\diff\) is fully faithful.  Equivalently,
\(\bar\diff^*\circ\bar\diff(V) \cong V\) for any smooth group
representation~\(V\) of~\(G\).

Thus the two differentiation functors \(\diff\)
and~\(\bar\diff\) from smooth representations of~\(G\) to Lie
algebra representations of~\(\g\) come together with two
integration functors \(\diff^*\) and~\(\bar\diff^*\) that map
Lie algebra representations of~\(\g\) to group representations
of~\(G\).

The integration functor~\(\diff^*\) is right adjoint to~\(\diff\).
That is, bounded \(G\)\nb-equivariant linear maps \(V\to \diff^*(W)\)
for smooth \(G\)\nb-representations~\(V\) correspond bijectively to
bounded \(\Univ(\g)\)-module homomorphisms from~\(\diff(V)\) to~\(W\).

The integration functor~\(\bar\diff^*\) is left adjoint
to~\(\bar\diff\).  That is, bounded \(G\)\nb-equivariant linear maps
\(\bar\diff^*(W)\to V\) for a smooth \(G\)\nb-representation~\(V\)
correspond bijectively to bounded \(\Univ(\g)\)-module homomorphisms
from~\(W\) to the roughening of~\(V\).

Thus our two integration functors both satisfy a universal property,
meaning that they are, in some sense, optimal ways to integrate a
\(\Univ(\g)\)-module.  The integration \(\diff^*(W)\) is the maximal
smooth \(G\)\nb-representation equipped with a \(\Univ(\g)\)-module
map \(W\to V\) in the sense that any other such~\(V\) maps
to~\(\diff^*(W)\).  And the integration \(\bar\diff^*(W)\) is the
minimal smooth \(G\)\nb-representation equipped with a
\(\Univ(\g)\)-module map \(W\to \Rough(V)\) in the sense that it maps
to any other such~\(V\).

However, these two universal properties are not compatible.  There is
usually no canonical map between \(\diff^*(W)\) and \(\bar\diff^*(W)\)
in either direction.

\end{document}